\documentclass[11pt]{article}
\usepackage{a4wide}

\usepackage[a4paper, margin=2.3cm]{geometry}
\usepackage{amsmath,amssymb,amsthm}
\usepackage{color}
\usepackage{parskip}
\usepackage{hyperref}
\usepackage{parskip}
\usepackage{setspace}
\usepackage{graphicx}
\usepackage{mathtools}
\usepackage[thinc]{esdiff}
 
\newtheorem{theorem}{Theorem}[section]

\newtheorem{lemma}[theorem]{Lemma}

\newtheorem*{definition*}{Definition}

\date{}
\begin{document}

\title{On the two-parameter Erd\H os-Falconer distance problem \\over finite fields}

\author{
Cl\'{e}ment Francois\thanks{ETH Zurich, Switzerland. Email: {\tt fclement@student.ethz.ch}}
\and
 Hossein Nassajian Mojarrad\thanks{Courant Institute, New York University. Email: {\tt sn2854@nyu.edu}. Supported by Swiss National Science Foundation grant P2ELP2-178313.}
\and
Duc Hiep Pham\thanks{University of Education, Vietnam National University, Hanoi. Email: {\tt phamduchiep@vnu.edu.vn}}
\and
Chun-Yen Shen \thanks{Department of Mathematics,  National Taiwan University. Email: {\tt cyshen@math.ntu.edu.tw}}
}

\date{}
\maketitle

\begin{abstract}
 Given $E \subseteq \mathbb{F}_q^d \times \mathbb{F}_q^d$, with the finite field $\mathbb{F}_q$ of order $q$ and the integer $d \ge 2$, we define the two-parameter distance set as $\Delta_{d, d}(E)=\left\{\left(\|x_1-y_1\|, \|x_2-y_2\|\right) : (x_1,x_2), (y_1,y_2) \in E \right\}$. Birklbauer and Iosevich (2017) proved that if $|E| \gg q^{\frac{3d+1}{2}}$, then $ |\Delta_{d, d}(E)| = q^2$. For the case of $d=2$, they showed that if $|E| \gg q^{\frac{10}{3}}$, then $ |\Delta_{2, 2}(E)| \gg q^2$. In this paper, we present extensions and improvements of these results. 
\end{abstract}


\section{Introduction}
The general Erd\H os distance problem asks to determine the number of distinct distances spanned by a finite set of points. In the Euclidean space, it is conjectured that for any finite set $E \subset \mathbb{R}^d$, $d\ge 2$, we have $|\Delta(E)| \gtrapprox |E|^{\frac{2}{d}}$, where $\Delta(E)=\{\|x-y\| : x,y \in E\}$. Here and throughout, $X\ll Y$ means that there exists $C>0$ such that $X\le CY$, and $X\lessapprox Y$ with the parameter $N$ means that for any $\varepsilon>0$, there exists $C_{\varepsilon}>0$ such that $X \le C_{\varepsilon}N^{\varepsilon}Y$.

The finite field analogue of the distance problem was first studied by Bourgrain, Katz, and Tao \cite{BKT04} over prime fields. In this setting, the Euclidean distance among any two points $\boldsymbol{x}=(x_1,\ldots, x_d),\boldsymbol{y}=(y_1,\ldots,y_d) \in \mathbb{F}_q^d$, the $d$-dimensional vector space over the finite field of order $q$, is defined as $\|\boldsymbol{x}-\boldsymbol{y} \|=\displaystyle\sum_{i=1}^d(x_i-y_i)^2\in\mathbb F_q$.
For prime fields $\mathbb{F}_p$ with $p\equiv -1 \pmod 4$, they showed that if 
$E \subset \mathbb{F}_p^2$ with $|E|=p^{\delta}$ for some $0<\delta<2$, then the distance set satisfies $|\Delta(E)| \gg |E|^{\frac{1}{2}+\varepsilon}$, for some $\varepsilon>0$ depending only on $\delta$. 

This bound does not hold in general for arbitrary finite fields $\mathbb{F}_q$ as shown by Iosevich and Rudnev \cite{IR07}. In this general setting, they considered the Erd\H os-Falconer distance problem to determine how large $E \subset \mathbb{F}_q^d$ needs to be so that $\Delta(E)$ spans all possible distances or at least a positive proportion of them. More precisely, they proved that $\Delta(E)=\mathbb{F}_q$ if $|E| > 2q^{\frac{d+1}{2}}$, where the exponent is sharp for odd $d$. It is conjectured that in even dimensions, the optimal exponent will be $\frac{d}{2}$. As a relaxed fractional variant for $d=2$, it was shown in \cite{CEHIK} that if $E \subseteq \mathbb{F}_q^2$ satisfies $|E| \gg q^{\frac{4}{3}}$, then $|\Delta(E)| \gg q$. A recent series of other improvements and generalizations on the Erd\H os-Falconer distance problem can be found in \cite{hieu, koh, claudiu, suk, thang}. 

Using Fourier analytic techniques, a two-parameter variant of the Erd\H os-Falconer distance problem for the Euclidean distance was studied by Birklbauer and Iosevich in \cite{BI17}. More precisely, given $E \subseteq \mathbb{F}_q^d \times \mathbb{F}_q^d$, where $d\ge 2$, define the two-parameter distance set as
\[
\Delta_{d, d}(E)=\left\{\left(\|x_1-y_1\|, \|x_2-y_2\|\right) : (x_1,x_2), (y_1,y_2) \in E \right\}
\subseteq \mathbb{F}_q \times \mathbb{F}_q. \]
They proved the following results. 

\begin{theorem}\label{res11}
Let $E$ be a subset in $\mathbb{F}_q^d \times \mathbb{F}_q^d$. If $|E| \gg q^{\frac{3d+1}{2}}$, then $ |\Delta_{d, d}(E)| = q^2$.
\end{theorem}

\begin{theorem}\label{res22}
Let $E$  be a subset in $\mathbb{F}_q^2 \times \mathbb{F}_q^2$. If $|E| \gg q^{\frac{10}{3}}$, then $ |\Delta_{2, 2}(E)| \gg q^2$.
\end{theorem}

In this short note, we provide an extension and an improvement of these results. Compared to the method in \cite{BI17}, our results are much elementary. 

For $\boldsymbol{x}=(x_1,\ldots,x_d),\boldsymbol{y}=(y_1,\ldots,y_d) \in \mathbb{F}_q^d$ and for an integer $s\ge 2$, we introduce
$$\|\boldsymbol{x}-\boldsymbol{y} \|_s:=\sum_{i=1}^d a_i(x_i-y_i)^s,$$ 
where $a_i \in \mathbb{F}_q$ with $a_i \neq 0$ for $i=1,\ldots, d$. For any set $E\subset \mathbb{F}_q^d\times \mathbb{F}_q^d$, define 
\[\Delta_{d, d}^s(E)=\left\{\left(\|x_1-y_1\|_s, \|x_2-y_2\|_s\right) : (x_1,x_2), (y_1,y_2) \in E \right\}.\]
Our first result reads as follows. 
\begin{theorem}\label{res1}
Let $E$ be a subset in $\mathbb{F}_q^d \times \mathbb{F}_q^d$. If $|E| \gg q^{\frac{3d+1}{2}}$, then $ |\Delta_{d, d}^s(E)| \gg q^2$.
\end{theorem}
It is worth mentioning that our method also works for the multi-parameter distance set defined for $E \subseteq \mathbb{F}_q^{d_1+\dots+d_k}$, but we do not discuss such extensions herein. 
For the case of $d=2$, we get an improved version of Theorem \ref{res22} for the usual distance function over prime fields.

\begin{theorem}\label{res3}
Let $E \subseteq \mathbb{F}_p^2 \times \mathbb{F}_p^2$. If $|E| \gg p^{\frac{13}{4}}$, then $ |\Delta_{2}(E)| \gg p^2$.
\end{theorem}

We note that the continuous version of Theorems \ref{res1} and \ref{res3} have been studied in \cite{ioo0, ioo1}. However, the authors do not know whether the method in this paper can be extended to that setting. Moreover, it follows from our approach that the conjecture exponent $\frac{d}{2}$ of the (one-parameter) distance problem would imply the sharp exponent for two-parameter analogue, namely, $\frac{3d}{2}$ for even dimensions. We refer the reader to \cite{BI17} for constructions and more discussions.
\section{Proof of Theorem \ref{res1}}
The following lemma  plays a key role in our proof for Theorem \ref{res1}. 
\begin{lemma}[Theorem 2.3, \cite{Vinh}]\label{lmm1}
Let $X, Y \subseteq \mathbb{F}_q^d$. Define $\Delta^s(X, Y)=\{\|x-y\|_s\colon x\in X, y\in Y\}$. If $|X||Y|\gg q^{d+1}$, then $|\Delta^s(X, Y)|\gg q$.
\end{lemma}
\begin{proof}[Proof of Theorem \ref{res1}]
By assumption, we have $|E|\ge  Cq^{d+\frac{d+1}{2}}$ for some constant $C>0$. For $y\in \mathbb{F}_q^d$, let $E_y:=\left\{x\in \mathbb F_q^d : (x, y)\in E\right\}$, and define 
\[
Y:=\left\{ y\in \mathbb{F}_q^d : ~ |E_y|> \frac{C}{2} q^{\frac{d+1}{2}} \right\}.
\]
We first show that $ |Y| \ge \frac{C}{2} q^{\frac{d+1}{2}}$. Note that
\[
|E|=\sum_{y \in Y} |E_y| + \sum_{y \in \mathbb{F}^d_q\setminus Y} |E_y| ~ \le ~
  q^d|Y| + \sum_{y \in \mathbb{F}^d_q\setminus Y} |E_y|,
\]
where the last inequality holds since $|E_y| \le q^d$ for $y\in \mathbb{F}_q^d$. Combining it with the assumption on $|E|$ gives the lower bound 
$
\sum_{y \in \mathbb{F}^d_q\setminus Y} |E_y| \ge Cq^{d+\frac{d+1}{2}} - q^d|Y|.
$
On the other hand, by definition, we have $|E_y| \le \frac{C}{2} q^{\frac{d+1}{2}}$ for $y \in \mathbb{F}^d_q\setminus Y$ yielding the upper bound 
$
\sum_{y \in \mathbb{F}^d_q\setminus Y} |E_y| \le \frac{C}{2} q^{d+\frac{d+1}{2}}
$.
Thus, these two bounds altogether give $Cq^{d+\frac{d+1}{2}} - q^d|Y| \le \frac{C}{2} q^{d+\frac{d+1}{2}}$, proving the claimed bound $|Y| \ge \frac{C}{2} q^{\frac{d+1}{2}}$.

In particular, Lemma \ref{lmm1} implies $|\Delta^s(Y,Y)|\gg q$, as $|Y||Y| \gg q^{d+1}$. On the other hand, for each $u\in \Delta^s(Y,Y)$, there are $z, t\in Y$ such that $\|z-t\|=u$. One has $|E_z|, |E_t|\gg q^{\frac{d+1}{2}}$, therefore, again by Lemma \ref{lmm1}, $|\Delta^s(E_z, E_t)|\gg q$. Furthermore, for $v\in \Delta^s(E_z, E_t)$, there are $x\in E_z$ and $y\in E_t$ satisfying $\|x-y\|_s=v$. Note that $x\in E_z$ and $y\in E_t$ mean that $(x,z), (y,t)\in E$. Thus, $(v,u)=(\|x-y\|_s, \|z-t\|_s)\in \Delta_{d, d}^s(E)$. From this, we conclude that $|\Delta_{d, d}^s(E)|\gg q|\Delta^s(Y,Y)|\gg q^2$, which completes the proof.
\end{proof}
\section{Proof of Theorem \ref{res3}}
To improve the exponent over prime fields $\mathbb{F}_p$, we strengthen Lemma \ref{lmm1} as follows. Following the proof of Theorem \ref{res1} with Lemma \ref{lmm2} below proves Theorem \ref{res3} then.
\begin{lemma}\label{lmm2}
Let $X, Y \subseteq \mathbb{F}_p^2$. If $|X|, |Y|\gg p^{\frac 5 4}$, then  $|\Delta(X, Y)|\gg p$.
\end{lemma}
\begin{proof}
It is clear that if $X' \subseteq X$ and $Y' \subseteq Y$, then $\Delta(X',Y')\subseteq \Delta(X,Y)$. Thus, without loss of generality, we may assume that $|X|=|Y|=N$ with $N\gg p^{\frac 5 4}$.
Let $Q$ be the number of quadruples $(x, y, x', y')\in X\times Y\times X\times Y$ such that $\|x-y\|=\|x'-y'\|$. It follows easily from the Cauchy-Schwarz inequality that
\[|\Delta(X, Y)|\gg \frac{|X|^2|Y|^2}{Q}.\]
Let $T$ be the number of triples $(x, y, y')\in X\times Y\times Y $ such that $\|x-y\|=\|x-y'\|$. By the Cauchy-Schwarz inequality again, one gets $Q\ll |X|\cdot T$. Next, we need to bound $T$. For this, denote $Z=X\cup Y$, then $N\le |Z|\le 2N$. Let $T'$ be the number of triples $(a, b, c)\in Z\times Z\times Z$ such that $\|a-b\|=\|a-c\|$. Obviously, one gets $T\le T'$. On the other hand, it was recently proved (see \cite[Theorem 4]{pham}) that 
\[T'\ll \frac{|Z|^3}{p}+p^{2/3}|Z|^{5/3}+p^{1/4}|Z|^2,\]
which gives  
\[T\ll \frac{N^3}{p}+p^{2/3}N^{5/3}+p^{1/4}N^2,\]
and then $T\ll \dfrac {N^3}p$ (since $N\gg p^{\frac 5 4}$).
Putting all bounds together we obtain
$$\dfrac{N^3}{|\Delta(X,Y)|}=\dfrac{|X||Y|^2}{|\Delta(X,Y)|}\ll \dfrac Q{|X|}\ll T\ll \dfrac {N^3}p,$$
or equivalently, $|\Delta(X,Y)|\gg p$, as required.
\end{proof}

\section*{Acknowledgment}
The authors would like to thank Thang Pham for sharing insights and new ideas.

\end{document}